\documentclass[oneside]{amsart}
\usepackage[utf8]{inputenc}

\usepackage{euscript}
\usepackage[T2A]{fontenc}
\usepackage[utf8]{inputenc}
\usepackage[normalem]{ulem}

\usepackage[english]{babel} 
\usepackage{amsmath}
\usepackage{mathtools}
\usepackage{amsfonts}
\usepackage{amssymb, amsthm}
\usepackage{dsfont}

\usepackage{hyperref}

\usepackage{color}

\usepackage{amsmath}
\usepackage{subfig}
\usepackage{geometry}
\usepackage{faktor}

\usepackage{multicol} 

\usepackage{graphicx}
\graphicspath{}
\DeclareGraphicsExtensions{.pdf,.png,.jpg} 

\newcommand{\mbbR}{\mathbb{R}}
\newcommand{\mbbC}{\mathbb{C}}

\newcommand{\mbbN}{\mathbb{N}}
\newcommand{\mbbZ}{\mathbb{Z}}

\newcommand{\mbbH}{\mathbb{H}}

\newcommand{\mbbF}{\mathbb{F}}

\newcommand{\eps}{\varepsilon}
\newcommand{\del}{\delta}

\newcommand{\suml}{\sum \limits}

\newcommand{\supl}{\sup \limits}

\newcommand{\abs}[1]{{\left| #1 \right|}}

\newcommand{\norm}[1]{{\| #1 \|}}

\newcommand{\wpart}[1]{{\left[ #1 \right]}}

\DeclareMathOperator{\diam}{diam}

\renewcommand{\le}{\leqslant}
\renewcommand{\ge}{\geqslant}

\theoremstyle{plain}
\newtheorem{thm}{Theorem}
\newtheorem*{thm*}{Theorem}
\newtheorem{claim}[thm]{Claim}
\newtheorem*{lm*}{Lemma}
\newtheorem{lm}[thm]{Lemma}
\newtheorem{cor}[thm]{Corollary}
\newtheorem{prop}[thm]{Proposition}

\newtheorem*{claim*}{Claim}

\theoremstyle{definition}

\theoremstyle{remark}
\newtheorem*{rem}{Remark}

\newcommand{\lr}[1]{{\left( #1\right)}}

\DeclareMathOperator{\Aut}{Aut}
\newcommand{\Adm}{\mathcal{A}dm}

\newcommand{\acts}[1]{\stackrel{{{{#1}}}}{{\curvearrowright}}}

\begin{document}

\title[Generic actions of almost complete growth]{{Non--existence of a universal zero entropy system via generic actions of almost complete growth}}
\author{Georgii Veprev}
\date{\today}
\thanks{The work is supported by Ministry of Science and Higher Education of the Russian Federation, agreement №  075-15-2022-287. The author acknowledges support of the Institut Henri Poincaré (UAR 839 CNRS-Sorbonne Université), and LabEx CARMIN (ANR-10-LABX-59-01).
}
\address{Leonhard Euler International Mathematical Institute in St. Petersburg, 
\newline 14th Line 29B, Vasilyevsky Island, St. Petersburg, 199178, Russia}
\email{egor.veprev@mail.ru}
\keywords{Generic action, universal system, scaling entropy, amenable group action, zero entropy, scaling entropy growth gap.}

\maketitle

\begin{abstract}
    We prove that a generic p.m.p. action of a countable amenable group~$G$ has scaling entropy that can not be dominated by a given rate of growth. {As a corollary, we obtain that there does not exist a topological action of~$G$ for which the set of {ergodic} invariant measures coincides with the set of all {ergodic} p.m.p. $G$--systems of entropy zero.} 
    
    We also prove that a generic action of a residually finite amenable group has scaling entropy that can not be bounded from below by a given sequence. We also show an example of an amenable group that has such lower bound for every free p.m.p. action.
\end{abstract}


\section{Introduction}

In this paper, we study generic p.m.p. actions of amenable groups. The main object we focus on is \emph{the scaling entropy} of an action --- the invariant of slow entropy type proposed by A.~Vershik in~\cite{V2010, V1,V3}. This invariant is based on the dynamics of \emph{measurable metrics} on the underlying measure space and reflects the asymptotic behavior of the minimal epsilon-net of the averaged metric. The scaling entropy invariant was studied in \cite{VPZ, PZ, Vep3, V3, Z1, Z2}. We will give all the necessary definitions in Section~\ref{sec_scaling}. 

It turns out that some properties of the scaling entropy of a generic action can be established. In particular, we show that its asymptotic behavior can not be bounded from above by any nontrivial bound. For the case of a singe transformation, similar results were obtained in~\cite{A, Vep3}. Together with the results from~\cite{Vep2}, this gives the negative answer to the Weiss' question about the existence of a universal zero entropy system (see \cite{S, Vep2}) for all amenable groups.

Also, we study lower bounds for the generic growth rate of scaling entropy. In the case of a residually finite group, the similar result holds true: there exists no non-constant lower bound for the scaling entropy of a generic action. 
However, it is not true in general. It turns out that there exist discrete amenable groups that have \emph{a scaling entropy growth gap} meaning that the scaling entropy of any free p.m.p. action of such a group has to grow faster than some fixed unbounded function. We show an example of such a group in Section~\ref{sec_egg}. 
Our example bases on the theory of growth in finite groups, in particular the growth theorem by H.~Helfgott~(see~\cite{H}) and its generalizations~from~\cite{PS}.

\subsection*{Acknowledgements} 
The author is sincerely grateful to his advisor Pavel Zatitskiy for his support and attention during this work and to   Andrei Alpeev, Ivan Mitrofanov, Valery Ryzhikov, and Markus Steenbock for fruitful discussions and communications. 
\subsection{Generic properties of group actions}
Descriptive set theory applied to group actions is a well studied concept in ergodic theory. We will give several definitions in order to set up notations. For more details follow survey~\cite{Ke} by A.~Kechris. Let~$\Gamma$ be a discrete countable group and $(X,\mu)$ a Lebesgue space. Let~$\Aut (X,\mu)$ be the group of all invertible measure--preserving transformations of $(X,\mu)$ endowed with the weak topology~$w$ with respect to which~$\Aut (X,\mu)$ is a Polish space. The set of all p.m.p. actions of~$\Gamma$ on~$(X,\mu)$ can be naturally identified  with the space~$A(\Gamma, X, \mu)$ of all homomorphisms from~$\Gamma$ to~$\Aut (X,\mu)$. Clearly, $A(\Gamma, X, \mu)$ is a closed subset of the space $\Aut (X,\mu)^\Gamma$ endowed with the product topology and, therefore, is Polish. Let us note that this topology is generated by the family~$\{U_{\gamma, a, \eps}(\alpha)\}_{\gamma \in \Gamma, a \subset X, \eps > 0}$  of open neighbourhoods as~prebase, where~$\alpha$ is a p.m.p. action of~$\Gamma$. Each  $U_{\gamma, a, \eps}(\alpha)$ consists of those $\beta \in A(\Gamma, X, \mu)$ that satisfy $\mu(\beta(\gamma) a \triangle \alpha(\gamma) a) < \eps$.

Every automorphism~$T \in \Aut (X,\mu)$ acts on~$A(\Gamma, X, \mu)$ by conjugation: $a \mapsto T a T^{-1}, \ a\in A(\Gamma, X, \mu)$. It is shown in~\cite{ForW} that the conjugacy class of every free ergodic action of an amenable group is dense in the weak topology of~$A(\Gamma, X, \mu)$.

We say that a set~$P$ of $\Gamma$--actions is \emph{meager} if its complement contains a dense $G_\del$~subset in~$A(\Gamma, X, \mu)$. We call~$P$ \emph{generic} (or \emph{comeager}) if it contains a dense~$G_\del$ subset itself. It is well known that, for example, the set of all ergodic free actions of a discrete amenable group~$\Gamma$ is generic, as well as the set of all actions with zero measure entropy~(see \cite{ForW, Ke}). 

\subsection{Universal systems}
{Universal dynamical systems appear in various contexts in lots of papers, see~\cite{DS, S, SW, Vep2, VZ2, W}, for example. The exact definition of universality varies from paper to paper. We will mainly follow the one given in~\cite{DS} by T.~Downarowicz and J.~Serafin.
Let~$G$ be an amenable group and~$X$ a metric compact space on which $G$~acts by homeomorphisms. A~topological system~$(X, G)$ is called \emph{universal} for some class~$\mathcal{S}$ of {ergodic} p.m.p. actions of~$G$ if the following two conditions are satisfied. 
For any {ergodic} $G$--invariant measure~$\mu$ on~$X$ the system $(X, \mu, G)$ belongs to $\mathcal{S}$ and for any $(Y,\nu, G) \in \mathcal{S}$ there exists a {$G$--invariant measure~$\mu$ on~$X$} such that  $(X, \mu, G)$ is measure--theoretically isomorphic to $(Y,\nu, G)$.
}

In view of the variational principle, the natural question about the existence of a universal system for the class of all zero entropy systems appears in~\cite{S}. This question goes back to B.~Weiss.

For the case of a single transformation, the negative answer to this question was given in~\cite{S} by J.~Serafin. His poof uses the notions of symbolic and measure--theoretic complexity of a dynamical system (see also~\cite{F}) and constructions of systems with rapidly growing measure--theoretic complexity.
In~\cite{Vep2}, the author extends Serafin's result to non-periodic amenable groups using the scaling entropy invariant, constructions of Vershik's automorphisms (see~\cite{VZ}), and coinduced actions. 

As a corollary of the results of the present paper, we answer the Weiss' question in full generality.
\begin{thm}\label{thm_univ}
    Every infinite countable discrete amenable group does not admit a universal zero-entropy system.
\end{thm}

\section{Slow entropy type invariants}
\subsection{Kushnirenko's sequential entropy}
As an intermediate step in our arguments we use the following \emph{sequential entropy} invariant introduced in~\cite{Kush}, or rather its generalized version from~\cite{R}. Let $P = \{P_n\}$ be a sequence of finite subsets in~$G$ and $G \acts{\alpha} (X, \mu)$ be a p.m.p. action of~$G$. For a measurable partition~$\xi$, define its sequential entropy as follows:
\begin{equation}
h_P(G,\xi) = \limsup_{n} \frac{1}{|P_n|} H\lr{\bigvee_{g\in P_n} g^{-1}\xi}.
\end{equation}
The sequential entropy along~$P$ of the action is the following supremum: 
\begin{equation}
h_P(X, \mu, G) = \sup_{\xi \colon H(\xi) < \infty} h_P(G, \xi).
\end{equation}

\subsection{Scaling entropy}\label{sec_scaling}
In this section, we give a brief introduction to the theory of scaling entropy. This invariant was introduced by A.~Vershik in his papers~\cite{V2010, V1,V3} and was further developed by F.~Petrov and P.~Zatitskiy in~\cite{VPZ, PZ, Z1, Z2}. The main Vershik's idea is to consider dynamical properties of functions of several variables, namely, measurable metrics and semimetrics (quasimetrics). 

Let us mention that the closely--related notions appear in several papers by S.~Ferenczi (\emph{measure--theoretic complexity}{, see, e.\,g.,~\cite{F}})  and A.~Katok \& J.-P.~Thouvenot (\emph{slow entropy}{, see~\cite{KT}}). We~refer the reader to survey~\cite{KKW} for details on these invariants.

Throughout this paper, we use the following notations. For two sequences~$\phi {=\{\phi(n)\}_n}$ and $\psi {=\{\psi(n)\}_n}$ of positive numbers, we write~$\phi \precsim \psi$ if the asymptotic relation $\phi(n) = O (\psi(n))$ is satisfied. We write~$\phi \asymp \psi$ if both inequalities~$\phi \precsim \psi$ and~$\psi \precsim \phi$ hold and~$\phi \prec \psi$ if~$\phi(n) = o(\psi(n))$. 

\subsubsection{Epsilon--entropy and measurable semimetrics}
Consider a measurable function $\rho\colon (X^2, \mu^2) \to [0, +\infty)$. We call $\rho$ a measurable semimetric if it is non--negative, symmetric, and satisfies the triangle inequality. 
For a positive $\eps$, the \emph{$\eps$--entropy} of the semimetric $\rho$ is defined in the following way. Let $k$ be the minimal positive integer such that the space~$X$ decomposes into a union of measurable subsets $X_0, X_1, \ldots, X_k$ with $\mu(X_0) < \eps$ and $\diam_\rho(X_i) < \eps$ for all $i>0$. Put
\begin{equation}
\mbbH_\eps(X, \mu, \rho) = \log_2 k.
\end{equation}
If there is no such finite $k$, we define $\mbbH_\eps(X, \mu, \rho) = + \infty$.

We call a semimetric \emph{admissible} if it is separable on some subset  of full measure. It turns out (see~\cite{VPZ}) that a semimetric is admissible if and only if its $\eps$--entropy is finite for any $\eps > 0$. 
In this paper, we consider only admissible semimetrics. A simple example of such a semimetric is so-called \emph{cut semimetric}~$\rho_\xi$ corresponding to a measurable partition~$\xi$ with finite Shannon entropy. That is, $\rho(x, y) = 0$ if both points $x,y \in X$ lie in the same cell of~$\xi$, and $\rho(x, y) = 1$ otherwise. 

{The space~$\Adm(X,\mu)$ of all summable admissible semimetrics is a convex cone in $L^1(X^2,\mu^2)$. Define the following \emph{$m$--norm} on a linear subspace of $L^1(X^2,\mu^2)$ containing $\Adm$:
\begin{equation}
\norm{f}_m = \inf \{\norm{\rho}_{L^1(X^2,\mu^2)}\ :\ \rho(x,y) \ge \abs{f(x,y)}, \ \mu^2\text{--a.s.}\},
\end{equation}
where the infimum is computed over all measurable semimetrics~$\rho$, see~\cite{VPZ, Z1} for details. }


\subsubsection{Scaling entropy of a group action}

Let~$G$ be a countable amenable group with some given F\o lner sequence $\lambda = \{F_n\}$ which we will call \emph{equipment} of the group~$G$. We will refer to the pair~$(G, \lambda)$ as an \emph{equipped group}. Let us remark right away that the scaling entropy invariant is well defined beyond amenable groups and F\o lner sequences. The only assumption one needs to make is the requirement of equipment to be~\emph{suitable}~(see \cite{Z2} for details), a sequence of increasing balls in a finitely generating group may be viewed as an example. However, we restrict our considerations to the case of amenable groups since we will deal only with them in this paper. 

Suppose that $G \acts{\alpha} \lr{X,\mu}$ is a p.m.p. action of~$G$ on a Lebesgue space $\lr{X,\mu}$. For a measurable semimetric~$\rho$ and an element $g \in G$, let $g^{-1}\rho$ denote a translation of~$\rho$:  $g^{-1}\rho(x,y) = \rho(gx, gy)$, where $x,y \in X$. Note, that if $\rho$ is admissible, then $g^{-1}\rho$ is admissible as well. A semimetric is said to be~\emph{generating} if all its translations together separate points of~the measure space up to a null~set. 

Consider the average of~$\rho$ over~$F_n$
\begin{equation}\label{eq_383647}
G^n_{av} \rho (x,y) = \frac{1}{\abs{F_n}}\suml_{g\in F_n} \rho(gx, gy), \qquad x,y \in X.
\end{equation}
We will also denote the same semimetric~\eqref{eq_383647} by the symbol $G^{F_n}_{av} \rho$ emphasizing the set of elements used to compute the average and $\alpha^{n}_{av} \rho$ emphasizing the action. 
Consider the following function:
\begin{equation}
\Phi_\rho(n, \eps) = \mbbH_\eps\lr{X, \mu, G^n_{av}\rho}.
\end{equation}
By definition, $\Phi_\rho(n,\eps)$ depends on $n$, $\eps$, and the semimetric $\rho$. However, its \emph{asymptotic behaviour~in~$n$} is supposed to be independent of~$\rho$ and~$\eps$ in some sense (see~\cite{V1, V3}). The strongest form of such independence corresponds to the following notion from~\cite{VPZ, Z1}. A sequence~$\{h_n\}$ is called a \emph{scaling entropy sequence} for~$\rho$ if $\Phi_\rho(n,\eps) \asymp h_n$ for all sufficiently small $\eps > 0$. P.~Zatitskiy showed in~\cite{Z1, Z2} that if a sequence~$\{h_n\}$ is a scaling entropy sequence for some generating~$\rho \in \Adm$, then it is also a scaling entropy sequence for any other such semimetric. Hence, the class of all scaling entropy sequences forms an invariant of the action. This invariant was studied in~\cite{VPZ, PZ, Vep1, V3, Z1, Z2}.

Although there are a lot of nice nontrivial cases where the scaling entropy sequence can be computed (see, e.\,g.,~\cite{Z2}), it does not always exist in this strong form, as shown in~\cite{Vep1}. In order to cover all of the cases, we use more general approach. We consider the set of functions mapping $\mbbN \times \mbbR_+$ to $\mbbR_+$ that decrease in their second arguments. Then we extend the relation~$\precsim$ to this set by setting for two functions~$\Phi$ and~$\Psi$
\begin{equation}\label{eq_po}
    \Phi \precsim \Psi \iff \forall \eps > 0\ \exists \delta > 0 \ \Phi(n, \eps) \precsim \Psi(n, \del).
\end{equation}
We call~$\Phi$ and~$\Psi$ equivalent (and write $\Phi \asymp \Psi$) if both relations $\Phi \precsim \Psi$ and $\Psi \precsim \Phi$ are satisfied.
The Zatitskiy's invariance theorem {from}~\cite{Z1, Z2} states that for any two generating semimetrics~$\rho$ and~$\omega$ in~$\Adm$ the following equivalence takes place: $\Phi_\rho \asymp \Phi_\omega$. Therefore, the equivalence class $\mathcal{H}(X,\mu, G, \lambda) = \wpart{\Phi_\rho}$ is an invariant of a p.m.p. action of an equipped group. We call this class \emph{the scaling entropy} of the action. We will also write  $\mathcal{H}(\alpha, \lambda)$ referring to the scaling entropy of a p.m.p. action~$\alpha$.

Also, we write~$\Phi \prec \Psi$ if there exists~$\delta >0$ such that for any~$\eps > 0$ we have $\Phi(n, \eps) \prec \Psi(n, \del)$. Clearly, relations~$\prec$ and~$\precsim$ agree with the equivalence relation~$\asymp$ and induce partial orders on the set of equivalence classes.

\section{Main results}
In this paper, we study the scaling entropy of a generic action. In Section~\ref{sec_upbound}, we look for p.m.p. actions whose scaling entropy can not be bounded by a given function. In~\cite{Vep2}, such actions are called \emph{actions of almost complete growth} and constructed \emph{explicitly} for any non--periodic amenable group~$G$. Such explicit constructions for general amenable groups are unknown. We prove that actions of almost complete growth are generic in the following sense. 

\begin{thm}\label{thm_upbound}
Let $G$~be a countable amenable group and $\lambda = \{F_n\}$ be a F\o lner sequence in~$G$. Let $\phi(n) = o(\abs{F_n})$ be a sequence of positive real numbers. Then the set of all zero-entropy ergodic p.m.p. actions of~$G$ that satisfy
\begin{equation}
\Phi(n, \eps) \not \precsim \phi(n) \text{ for sufficiently small $\eps >0$},
\end{equation}
where $\Phi \in \mathcal{H}(\alpha, \lambda)$, contains a dense $G_\del$-subset in~$A(G, X, \mu)$.
\end{thm}

We also study lower bounds for the scaling entropy of a generic action. For any residually finite group the similar result holds true.
\begin{thm}\label{thm_resfin}
    Let~$G$ be an infinite countable residually finite amenable group with a F\o lner sequence~$\lambda$ and~$\phi(n)$ a function with $\lim_n \phi(n) = \infty$. Then the set of all p.m.p. $G$--actions satisfying $\mathcal{H}(\alpha, \lambda) \succ \phi$ is meager.
\end{thm}
However, there exist groups with the property that the scaling entropy of any free p.m.p. action has to grow faster than a given function. We call this property \emph{a scaling entropy growth gap}. In Section~\ref{sec_egg}, we give an example of such a group (Theorem~\ref{thm_egg}) and prove that this property does not depend on the choice of F\o lner sequence. 

\section{Generic actions of almost complete growth}\label{sec_upbound}
\subsection{Sequential entropy of generic actions}
In \cite{R}, V. Ryzhikov proves that the set of all automorphisms~$T \in \Aut(X,\mu)$ such that $h_P(T) = +\infty$ contains a dense $G_\delta$ subset of $ \Aut(X,\mu)$ provided $\min\{|x - y| \colon x, y \in P_n,\ x\not = y\} $ goes to infinity. We use this approach to obtain the following proposition.

\begin{prop}\label{prop_kush}
Let $G$ be a countable amenable group and $\{P_n^l\}_{n = 1, \ldots, \infty}^{l = 1, \ldots, k_n}$ be a family of finite subsets of $G$ such that for any finite $K \subset G$, any sufficiently large $n$, and $g,h \in P_n^l$ we have $gh^{-1} \not \in K$ for all $l = 1, \ldots, k_n$. Then the set of all actions of $G$ on $(X, \mu)$ satisfying 
\begin{equation}\label{eq_365686765}
\supl_\xi \limsup\limits_n \min\limits_{l = 1, \ldots, k_n} \frac{1}{\abs{P_n^l}} H\lr{\bigvee_{g \in P_n^l} g^{-1}\xi} = +\infty,
\end{equation}
where supremum is computed over all finite measurable partitions, is comeager.
\end{prop}
\begin{proof}
Let $\{\xi_i\}_{i = 1}^{\infty}$ be a dense family of finite measurable partitions of $(X,\mu)$. Consider a countable dense family~$\{\alpha_q\}_{q\in I}$ of Bernoulli $G$--actions. Such family exists in the conjugacy class of any Bernoulli action. For any $q \in I$ and any $k > 0$ there exists some $j_{k,q} > k$ such that for any $j \ge j_{k,q}$  
\begin{equation}\label{thm_kush_eq1}
    R(\alpha_q, \xi_i, j) =
    \min\limits_{l = 1, \ldots, k_j} \frac{1}{\abs{P_j^l}} H\lr{\bigvee_{g \in P_j^l} \alpha_q(g)^{-1}\xi_i} > H(\xi_i) - \frac{1}{k}, \quad i = 1, \ldots, k.
\end{equation}
Indeed, since $\alpha_q$ is Bernoulli, every partition~$\xi_i$ can be approximated by a cylindrical partition whose translations over $P_j^l$ are independent for sufficiently large~$j$ and $l = 1, \ldots, k_j$ due to our assumptions on family~$\{P_j^l\}$. 
Since the function $R(\alpha, \xi_i, j_{k,q})$ is weakly continuous in~$\alpha$,
the set~$U_{k,q}$ of all p.m.p.~actions~${\alpha \in }A(G, X, \mu)$ satisfying $R(\alpha, \xi_i, j_{k,q}) > H(\xi_i) - \frac{1}{k}$ for every $i = 1, \ldots, k$ is weakly open. Consider then the following set
\begin{equation}
W = \bigcap\limits_k\bigcup\limits_q U_{k,q}.
\end{equation}
Clearly, $W$ is $G_\del$, contains every~$\alpha_q$, and, therefore, is dense. Every action in~$W$ satisfies the desired condition~\eqref{eq_365686765}. Indeed, for $\alpha \in W$, for every~$i > 0$ and every~$k > i$, there is some $q(k)$ such that $R(\alpha, \xi_i, j_{k, q(k)}) > H(\xi_i) - \frac{1}{k}$. Hence, $\limsup_n R(\alpha, \xi_i, n) \ge H(\xi_i)$ and, since $\{\xi_i\}$ is dense, $\sup_\xi \limsup_n R(\alpha, \xi, n) = +\infty$.
\end{proof}
\subsection{Proof of Theorem~\ref{thm_upbound} and non--existence of a universal zero-entropy system}
In this section, we prove Theorem~\ref{thm_upbound} and obtain Theorem~\ref{thm_univ} as its corollary. We find it easier to verify desired generic properties for sequential entropy first and then transfer them to scaling entropy having certain relations between these two invariants in hand. Finally, the scaling entropy invariant has deeper connections to the topological entropy that play its role in proving the non--existence of a universal zero entropy system. 
The direct proof without sequential entropy also seems possible. It would, however, involve some technical details that we would like to avoid.

We proceed with the following proposition connecting sequential entropy in the sense of Proposition~\ref{prop_kush} to the scaling entropy of the action. 
\begin{prop}\label{prop_rel}
Consider for every integer~$n$ a family $\{P_n^l\}^{l = 1, \ldots, k_n}$ of finite disjoint subsets of a~countable group~$G$ such that $F_n = \cup_{l=1}^{k_n} P_n^l$ is a F\o lner sequence. Assume that for some p.m.p. action~$\alpha$ of~$G$
\begin{equation}\label{rel_kush}
\supl_\xi \limsup\limits_n \min\limits_{l = 1, \ldots, k_n} \frac{1}{\abs{P_n^l}} H\lr{\bigvee_{g \in P_n^l} g^{-1}\xi} > 0.
\end{equation}
Then for any $\Phi \in \mathcal{H}(\alpha, \lambda)$, where $\lambda = \{F_n\}$,
\begin{equation}
\Phi(n, \eps) \not \prec \frac{\abs{F_n}}{k_n}
\end{equation}
for any sufficiently small $\eps > 0$.
\end{prop}

\begin{proof}
Consider a finite partition~$\xi$ satisfying relation~\eqref{rel_kush}, and let $c$ be the corresponding value of the left hand side. Let $\rho_\xi$ be the corresponding cut semimetric.
Let $\tilde F_n \subset F_n$ be the union of those~$P_n^l$ that satisfy 
\begin{equation}\label{eq_457569453}
    \abs{P_n^l} > \frac{|F_n|}{2k_n}.
\end{equation}
Let $L_n$ be the set of corresponding indices $l$-s. One may easily see that $|{\tilde F_n}| \ge \frac{1}{2}\abs{F_n}$. Hence, $G_{av}^{\tilde F_n}\rho_\xi(x,y) \le 2 G_{av}^{F_n}\rho_\xi(x,y)$ for any $x, y \in X$. Therefore,
\begin{equation}
\mbbH_\eps(X, \mu, G_{av}^{F_n}\rho_\xi) \ge \mbbH_{2\eps}(X, \mu, G_{av}^{\tilde F_n}\rho_\xi).
\end{equation}
Then we use the following lemma, which is proved in~\cite{PZ}, to estimate $\mbbH_{2\eps}(X, \mu, G_{av}^{\tilde F_n}\rho_\xi)$ from below.
\begin{lm}\label{lm_lowerbound}
        Let $\rho_1, \ldots, \rho_k$ be admissible semimetrics on~$\lr{X,\mu}$ such that $\rho_i(x,y) \le 1$ for all $i \le k,\ x, y \in X$. Let $\tilde\rho = \frac{1}{k} (\rho_1+ \ldots+ \rho_k)$. Then there exists some $m \le k$ such that
            \begin{equation}
            \mbbH_{2\sqrt{\eps}}\lr{X,\mu, \rho_m} \le 
            \mbbH_{\eps}\lr{X,\mu, \tilde \rho}.
            \end{equation}
    \end{lm}
It is easy to see that the same result holds for a convex combination $\tilde \rho = \sum_{i} \alpha_i \rho_i$, where $\alpha_i > 0$,  $\alpha_1 + \ldots+ \alpha_k =1$. In our case, we have 
\begin{equation}
G_{av}^{\tilde F_n}\rho_\xi = \suml_{l\in L_n} \frac{\abs{P_n^l}}{|{\tilde F_n}|} G_{av}^{P_n^l}\rho_\xi.
\end{equation}
Thus, there exists some $l \in L_n$ such that $\mbbH_{2\eps}(X, \mu, G_{av}^{\tilde F_n}\rho_\xi) \ge \mbbH_{2\sqrt{2\eps}}(X, \mu, G_{av}^{P_n^{l}}\rho_\xi)$. 
Suppose that~$n$ is such that
\begin{equation}
\min\limits_{l = 1, \ldots, k_n} \frac{1}{\abs{P_n^l}} H\lr{\bigvee_{g \in P_n^l} g^{-1}\xi} > \frac{c}{2}.
\end{equation}
We use the following lemma from~\cite{Z1} that connects $\eps$--entropy to the classical Shannon entropy.
\begin{lm}\label{lm_partitions}
Let $m,k \in \mbbN$ and $\{\xi_i\}_{i=1}^k$ be a family of finite measurable partitions each having no more than~$m$ cells.  Let $\xi = \bigvee_{i=1}^k\xi_i$ be the common refinement of these partitions and $\rho = \frac{1}{k} \sum_{i=1}^k\rho_{\xi_i}$ be the average of corresponding semimetrics. Then for any $\eps \in (0, \frac{1}{2})$ the following estimate holds:
        \begin{equation}
        \frac{H(\xi)}{k} \le \frac{\mbbH_\eps(X, \mu, \rho)}{k} + 2\eps \log m - \eps \log \eps - (1 - \eps)\log(1 - \eps) + \frac{1}{k}.
        \end{equation}
\end{lm}
Let $m = \abs{\xi}, \ \xi_g = g^{-1}\xi$, where $g\in P_n^l$. According to Lemma~\ref{lm_partitions}, we have
\begin{equation}
\mbbH_{2\sqrt{2\eps}}(X, \mu, G_{av}^{P_n^l}\rho_\xi)  \ge
\mbbH_{4\sqrt{\eps}}(X, \mu, G_{av}^{P_n^l}\rho_\xi)
> \abs{P_n^l}\left(\frac{c}{2} -8\sqrt{\eps} \log m - \delta(4\sqrt{\eps})\right) - 1,
\end{equation}
where $\delta(\eps) = -2\eps \log \eps - 2(1 - \eps)\log(1 - \eps)$, which tends to zero when $\eps$ goes to zero. 
Then, choosing~$\eps$ sufficiently small depending only on $c$ and $m = \abs{\xi}$, we obtain $\mbbH_{4\sqrt{\eps}}(X, \mu, G_{av}^{P_n^l}\rho_\xi) > \frac{c}{4} \abs{P_n^l}$. Since $\abs{P_n^l} > \frac{|F_n|}{2k_n}$ by assumption~\eqref{eq_457569453}, we obtain
\begin{equation}
\mbbH_\eps(X, \mu, G_{av}^{F_n}\rho_\xi) \ge \mbbH_{4\sqrt{\eps}}(X, \mu, G_{av}^{P_n^l}\rho_\xi) > \frac{c}{4} \abs{P_n^l} > \frac{c}{8} \cdot \frac{|F_n|}{k_n}.
\end{equation}
Thus, at least along some subsequence $\mbbH_\eps(X, \mu, G_{av}^{F_n}\rho_\xi) \gtrsim \frac{|F_n|}{k_n}$, and that completes the proof.
\end{proof}

\begin{proof}[Proof of Theorem~\ref{thm_upbound}]
It suffices to construct a family $\{P_n^l\}_{n = 1, \ldots, \infty}^{l = 1, \ldots, k_n}$ of finite subsets of~$G$ satisfying assumptions of Proposition~\ref{prop_kush} and such that $\frac{\abs{F_n}}{k_n} \succ \phi(n)$. Then the desired result follows from Proposition~\ref{prop_rel}. 

Let~$K$ be a finite subset of~$G$. Consider a locally finite graph $\Gamma_K = (G, E_K)$, where $(g, h)$ belongs to $E_K$ if and only if either $gh^{-1} \in K$ or $hg^{-1} \in K$. Clearly, the degree of each vertex in~$\Gamma_K$ does not exceed~$2\abs{K}$. Therefore, there exists a proper vertex coloring of $\Gamma_K$ into $r_K = 2\abs{K} + 1$ colors, that~is, a partition of all vertices into $r_K$~parts such that any two adjacent vertices belong to different parts.  
Indeed, one may color vertices one by one; each time there is at least one color available since no more than~$2|K|$  colors are prohibited. Hence, we obtain a decomposition $G = \bigcup_{l =1}^{r_K} C_K^l$, where $C_K^l$ are mutually disjoint and $gh^{-1} \not \in K$ for any $l\le r_k$ and any $g, h \in C_K^l$.

Take a sequence of increasing finite subsets exhausting the entire group: $K_1 \subset K_2 \subset \ldots  \subset \bigcup K_i = G$. Now let~$i(n)$ be a non-decreasing sequence of positive integer parameters with $\lim i(n) = +\infty$, which we will define  later. Put
\begin{equation}
 P_n^l = F_n \cap C_{K_{i(n)}}^l, \quad l = 1, \ldots, r_{K_{i(n)}}.    
\end{equation}
Clearly, the family~$\{P_n^l\}$ satisfies the assumptions of Proposition~\ref{prop_kush}. Since by assumptions of Theorem~\ref{thm_upbound} the sequence $\frac{|{F_n}|}{\phi(n)}$ goes to infinity, we can chose a piecewise constant  sequence~$i(n)$, also tending to infinity, such that $k_n = r_{K_{i(n)}} \prec \frac{|F_n|}{\phi(n)}$. Therefore, $\frac{|F_n|}{k_n} \succ \phi(n)$ as desired. 
\end{proof}

Of course, the genericity implies existence, and we obtain the following corollary. 
\begin{cor}
    Any countable amenable group admits actions of almost complete growth with respect to any F\o lner sequence. 
\end{cor}
To finish the proof of Theorem~\ref{thm_univ}, it only remains to recall the following theorem proved in~\cite{Vep2}.
\begin{thm}\label{thm_univ1}
     Suppose that an amenable group~$G$ admits {ergodic} actions of almost complete growth for some F\o lner equipment. Then $G$~does not have a universal zero-entropy system.
\end{thm}
As a consequence we obtain that there does not exist a universal zero entropy system for any countable amenable group; that~is, the Weiss' question is solved in full generality. 

\section{Generic lower bounds and scaling entropy growth gap}
Let us recall that a unitary representation of a discrete group is called \emph{compact} if every vector has a precompact orbit. A p.m.p. action is called compact if the corresponding Koopman representation is compact. It is shown in~\cite{VPZ} that  for the group~$\mbbZ$ this property is equivalent to the \emph{boundedness} of the scaling entropy. In fact, the same proof works for the case of an amenable group with F\o lner equipment. One may see~\cite{YZZ}, for instance.  

\subsection{Absence of a generic lower bound for residually finite groups}\label{sec_lowbound} 
Any countable residually finite amenable group admits a compact free p.m.p. action and, therefore, has an action with bounded scaling entropy, that is, the scaling entropy with the slowest growth possible. Indeed, one may consider an infinite product of finite approximations endowed with the natural product measure. The reverse implication is not true in general: the group of all dyadic rotations of a unit circle, for example, is not residually finite and, nevertheless, has a compact free action. However, the converse implication is true for finitely generated groups.

\begin{claim}\label{claim_compact}
A finitely generated group admits a compact free action if and only if it is residually finite.
\end{claim}
\begin{proof}

Let~$\alpha$ be a compact p.m.p. action of a group~$G$ and~$\pi$ its Koopman representation. Any compact action of a discrete group decomposes into direct sum of finite-dimensional representations (see, e.\,g.,~\cite{KL}). Therefore, $\pi = \bigoplus \tau_i$ and $\dim \tau_i = n_i < \infty$. The full image of~$\tau_i$ is a finitely generated subgroup in $GL_{n_i}(\mathbb{C})$. Hence, $ \tau_i  (G)$ is residually finite due to Maltcev's theorem.    
Since the action~$\alpha$ is free, the group~$G$ is residually finite as well.
\end{proof}

\begin{thm}\label{thm_lowbound}
Let~$G \acts{\alpha} (X, \mu)$ be a free ergodic p.m.p. action of an amenable group~$G$ and $\lambda = \{F_n\}$ be a F\o lner sequence in~$G$. Let~$\phi(n)$ be a non-negative function satisfying $\phi\succ \mathcal{H}(\alpha, \lambda)$. Then the set of all free p.m.p. actions~$\beta$ of~$G$ with $\mathcal{H}(\beta, \lambda) \succ \phi$ is meager. 
\end{thm}

Applying Theorem~\ref{thm_lowbound} to {a compact action of a} residually finite amenable group, we obtain Theorem~\ref{thm_resfin}.

\begin{proof}
Consider a dense sequence of finite measurable partitions~$\{\xi_i\}_{i = 1}^{\infty}$ of $(X,\mu)$ and a measurable metric~$\rho = \sum_{i=1}^\infty \frac{1}{2^i}\rho_{\xi_i}$. Let~$\{\alpha_q\}$ be a countable dense family of $G$--actions from the conjugacy class of~$\alpha$. Also, fix a monotone sequence~$\{\eps_r\}$ of positive numbers tending to zero. 
For any~$q$ and~$k$ there exists a~$j_{k,q}$ such that 
\begin{equation}
     \mbbH_{\frac{\eps_k}{4}}(X, \mu, (\alpha_q)_{av}^{j_{k,q}} \rho) < \frac{1}{k}\phi(j_{k,q}).
\end{equation}
Consider a neighbourhood~$U_{k,q}$ of~$\alpha_q$ such that for every~$\beta \in U_{k,q}$ the following holds true 
\begin{equation}\label{eq_938687634}
     \mbbH_{\eps_k}(X, \mu, \beta_{av}^{j_{k,q}} \rho) < \frac{1}{k}\phi(j_{k,q}).
\end{equation}
Such~$U_{k,q}$ does indeed exist due to the following lemma from~\cite{Z1}.
\begin{lm}\label{lm_mnorm}
Assume that $\norm{\rho_1 - \rho_2}_m < \eps^2/32$, where $\rho_1, \rho_2 \in \Adm(X,\mu)$ and $\eps > 0$. Then the inequality $\mbbH_\eps(X,\mu, \rho_1) < \mbbH_{\eps/4}(X,\mu, \rho_2)$ holds true.
\end{lm}
Indeed, having Lemma~\ref{lm_mnorm} in hand, we can uniformly approximate~$\rho$ by a partial sum $\sum_{i=1}^r \frac{1}{2^i}\rho_{\xi_i}$. Then the desired inequality~\eqref{eq_938687634} is achieved provided $\mu(\beta(g^{-1})C \triangle\alpha_q(g^{-1})C)$ is sufficiently small for every set~$C$ being a cell of~$\xi_i$, where $i \le r,\ g \in F_{j_{k,q}}$.

Then consider the following $G_\del$--set:
\begin{equation}
W = \bigcap\limits_k\bigcup\limits_q U_{k,q}.
\end{equation}
Consider any~$\beta\in W$ and any integer number~$r$. Then for any~$k > r$ there exists~$q_k$ such that
\begin{equation}\label{eq_240987965}
     \mbbH_{\eps_r}(X, \mu, \beta_{av}^{j_{k,q_k}} \rho) \le \mbbH_{\eps_k}(X, \mu, \beta_{av}^{j_{k,q_k}} \rho)< \frac{1}{k}\phi(j_{k,q_k}).
\end{equation}
Since~$\rho$ is an admissible metric, the function $\Phi(n, \eps) = \mbbH_{\eps}(X, \mu, \beta_{av}^n \rho)$ belongs to the scaling entropy class $\mathcal{H}(\beta, \lambda)$. Therefore, any $\beta \in W$ satisfies $\mathcal{H}(\beta, \lambda) \not\succ \phi(n)$ due to inequality~\eqref{eq_240987965}.
\end{proof}

\begin{rem}
We did not really use the F\o lner property of equipment~$\lambda$ while proving Theorem~\ref{thm_upbound} and Theorem~\ref{thm_lowbound}. The same results are also valid if we assume~$\lambda$ to be only \emph{suitable} (see~\cite{Z2}). The essential part is, however, that the group itself is amenable. It is unknown to the author if there are similar results for non-amenable groups.
\end{rem}

\subsection{Example of a group with a scaling entropy growth gap}\label{sec_egg}
In view of Section~\ref{sec_lowbound} and Theorem~\ref{thm_resfin}, one may wonder if it is always the case that the scaling entropy of a generic action grows arbitrarily  slow (along a subsequence, of course). We already know that it is true provided the group possesses a compact free action, but it  is unclear for groups without such actions. We say that a group~$G$ has \emph{a scaling entropy growth gap} with respect to equipment~$\lambda$ if there exists a function~$\phi(n)$ tending to infinity such that $\mathcal{H}(\alpha, \lambda) \succsim \phi$ for every free p.m.p. action~$\alpha$ of the group~$G$. In this section we show that there exists a group with a scaling entropy growth gap.  

 Let $G = SL(2, \overline{\mbbF}_p)$ be the group of all $2 \times 2$--matrices with determinant~$1$ over the algebraic closure of a finite field~$\mbbF_p$, where $p > 2$ is a prime number. Clearly, $G$~is countable, and it  can be presented as a union of increasing finite subgroups 
 $
 G = \bigcup_{n = 1}^\infty G_n 
 $,
 where each $G_n =  SL(2, {\mbbF_{q_n}})$ and~$\mbbF_{q_n}$ is a finite extension of $\mbbF_{q_{n-1}}$.

We will use the following Growth Theorem, initially proved in~\cite{H} by H.~Helfgott for~$SL(2, {\mbbF}_p)$ and then generalized to the following result (see~\cite{PS}). 
\begin{thm}\label{thm_growth}
    Let~$L$ be a finite simple group of Lie type of rank~$r$ and~$A$ a generating set of~$L$. Then either $A^3 = L$ or $|A^3| > c |A|^{1+\del}$, where~$c$ and~$\del$ depend only on~$r$.
\end{thm}

\begin{thm}\label{thm_egg}
    The group~$G = SL(2, \overline{\mbbF}_p)$ with equipment $\lambda = \{G_n\}$ admits scaling entropy growth gap. The function $\phi(n) = \log (q_n)$ is the desired lower bound. 
\end{thm}

\begin{proof}
Consider a free $p.m.p.$ action $G \acts{} (X, \mu)$. Take some non-trivial element~$g_0$ from $G_1 = SL(2, {\mbbF_{p}})$, let us take $g_0 = \big(\begin{smallmatrix}
  1 & 1\\
  0 & 1
\end{smallmatrix}\big)$, for instance. Since $g_0$~has order~$p$ and the action is free, there exists a measurable partition~$\xi$ of~$(X,\mu)$ into $p$~cells such that $\xi(x) \not = \xi((g_0)^i x)$ for every $i = 1, \ldots, p-1$. That~is, each cell of $\xi$ contains exactly one point from each $g_0$--orbit. 
Let~$\rho_\xi$ be the cut semimetric corresponding to~$\xi$.

Suppose that $\mbbH_{\eps^2}(X,\mu, G^n_{av}\rho_\xi) < \log k$ and let $X_0, X_1, \ldots, X_k$ be the corresponding decomposition. Since~$G_n$ is finite, the measure space decomposes as $(G_n, \nu) \times (Y, \eta)$, where the action of~$G_n$ preserves the second component. Since the exceptional set $X_0$ has measure less than $\eps^2$, the $\eta$-measure of those~$y$-s that satisfy $|G_n \times \{y\} \cap X_0| > \eps|G_n|$ is less than~$\eps$. The restriction of~$G^n_{av}\rho_\xi$ to each $G_n$--orbit is $G_n$--invariant and can be obtained by averaging the restriction of~$\rho_\xi$. The~restriction of~$\rho_\xi$ to a  $G_n$--orbit corresponds to its partition into $p$~parts of equal size. Hence, the restriction of~$\rho_\xi$ has mean value at least $\frac{1}{2}$ as well as its average since averaging preserves $L^1$--norm. 
All of the above implies that there exits at least one $G_n$--orbit with an invariant metric that has $\eps$--entropy (with respect to uniform measure) less than $\log k$ and $L^1$--norm of at least~$\frac{1}{2}$. It suffices to prove the following.
\begin{claim}\label{claim_sl2}
Let $\rho$ be a left--invariant semimetric on $SL(2, {\mbbF_{q}})$ with diameter greater than~$3\eps$, where $\eps \in (1,\frac{1}{2})$. Then $\mbbH_\eps(SL(2, {\mbbF_{q}}), \nu, \rho) \ge c\log q $, where $\nu$~is the uniform probability measure and $c$ is an absolute constant.
\end{claim}
Indeed, we can identify the orbit that we found above with the group $SL(2, {\mbbF_{q_n}})$ with the left--invariant semimetric which has diameter at least $\frac{1}{2}$.  Applying Claim~\ref{claim_sl2}, we obtain $\log k \ge c\log q_n$ and complete the proof.

In order to prove Claim~\ref{claim_sl2}, we can assume that $q$ is sufficiently large depending only on~$\delta$, which is an absolute constant since the rank~$r = 2$. Also, assume that $\mbbH_\eps(SL(2, {\mbbF_{q}}), \nu, \rho) < c\log q $. Then at most~$q^c$ balls of radius~$\eps$ cover the entire group except a part of size $\eps|SL(2, \mbbF_{q})|$. Since the semimetric~$\rho$ is left--invariant, all balls with the same radius have the same size. Therefore, the size of each ball is at least $\frac{1}{2 q^c} |SL(2, \mbbF_{q})|$. Let~$B = B(\eps)$ be the ball of radius $\eps$ with center at identity. Since the diameter of the group is greater than~$3\eps$, the product $B(\eps)\cdot B(\eps) \cdot B(\eps) \subset B(3\eps)$ does not cover the whole group. Therefore, due to the Growth Theorem~\ref{thm_growth},  we have two options. Either $|BBB| \ge |B|^{1 + \del}$, or the ball~$B$ does not generate $SL(2, {\mbbF_{q}})$. 
In the first case, we have 
\begin{equation}
|SL(2, \mbbF_{q})| \ge |BBB| \ge \frac{1}{2^{1+\del} q^{c(1+\del)}} |SL(2, \mbbF_{q})|^{1 + \del}.
\end{equation}
Hence, 
\begin{equation}
q^{c(1 + \del)} \ge \frac{1}{2^{1 + \del}} |SL(2, \mbbF_{q})|^\del \ge \frac{1}{2^{1 + \del}} q^{\del}
\end{equation}
and, therefore, $c > \frac{\del}{2 + 2\del}$ provided~$q$ is sufficiently large. 

In the last case, the subgroup~$H$ generated by~$B$ contains at least $\frac{1}{2 q^c} |SL(2, \mbbF_{q})|$ elements and, hence, has index smaller than~$2q^c$. Note that all non-trivial irreducible representations of $SL(2, {\mbbF_{q}})$ over~$\mbbC$ have dimension of at least~$\frac{q-1}{2}$, see~\cite{J, Sch}. However, the unitary representation corresponding to the permutation action of $SL(2, {\mbbF_{q}})$ on $SL(2, {\mbbF_{q}}) / H$ has dimension less than~$2q^c$ implying that~$c > \frac{1}{2}$. 

In both cases, we have $c > \frac{\del}{2 + 2\del}$, therefore, $\mbbH_\eps(SL(2, {\mbbF_{q}}), \nu, \rho) \ge \frac{\del}{2 + 2\del} \log q$, and the claim is proved.  
\end{proof}

Notably, the logarithmic bound from Theorem~\ref{thm_egg} \emph{is sharp}. For any group~$G$ that can be presented as an increasing union of finite groups $G_n$, one can define the following p.m.p. action. Let $C_n = \{g_n^j\}_{j = 1}^{k_n}$ be the  set of right coset representatives of $G_{n-1} \backslash G_n$ endowed with uniform measure $\mu_n$. Each finite product space $\prod_{i=1}^n (C_i, \mu_i)$ can be identified with the group~$G_n$ with the uniform measure and, therefore, carries a p.m.p. action of~$G_n$. Since these actions of $G_n$-s agree, we obtain a p.m.p. action of~$G$ on the whole product space $(X, \mu) = \prod_{i=1}^\infty (C_i, \mu_i)$, where each subgroup~$G_n$ preserves all the components starting from~$n+1$. 

Take $\rho = \sum_i 2^{-i} \rho_i$, where each~$\rho_i$ is the cut semimetric distinguishing first $i$~components. Clearly, $\rho$~is an admissible metric, and for any $n > r$ the average $G_{av}^n \sum_{i < r} 2^{-i} \rho_i$ does not depend on coordinates starting from~$n+1$. Therefore, there exists a partition into $|G_n|$ cells, each of which has diameter zero with respect to $G_{av}^n \sum_{i < r} 2^{-i} \rho_i$. Hence, for any positive~$\eps$, the $\eps$--entropy of~$G_{av}^n \rho$ is bounded from above by~$\log|G_n|$ for sufficiently large~$n$. For the case of $G = SL(2, \overline{\mbbF}_p)$, we have $q< SL(2, {\mbbF_q}) < q^4$. Hence, $\log|G_n| \asymp \log q_n$, and the bound is sharp. 

Also, looking through the proof of Theorem~\ref{thm_egg}, one may see \emph{a stronger alternative}. For every (not  necessarily free) p.m.p. action of $SL(2, \overline{\mbbF}_p)$, its scaling entropy is either bounded  or grows at least as fast as $\phi(n) = \log (q_n)$. 

Let us also mention that the scaling entropy growth gap property does not depend on which F\o lner sequence we choose.
\begin{prop}
    The property of having scaling entropy growth gap does not depend on the choice of F\o lner equipment.
\end{prop}
\begin{proof}
Assume that a group~$G$ has a scaling entropy growth gap with respect to a F\o lner sequence~$\{F_n\}$. Let $\phi(n)$ be a corresponding bound and~$\{W_n\}$ another F\o lner sequence in~$G$. 

For any integer~$n$, there exits some~$k_n$ such that for any $r > k_n$ the inequality $|F_n W_{r} \triangle W_{r}| < 2^{-n} | W_{r}|$ is satisfied.
Let $(X, \mu, G)$ be a free p.m.p. action of~$G$ and~$\rho$ a measurable metric bounded from above by one almost everywhere. Then
\begin{equation}
\frac{1}{|W_{r}|}\suml_{g \in W_{r}} g^{-1} \frac{1}{|F_n|}\suml_{h \in F_n} h^{-1} \rho  
\le \frac{1}{|W_r|} \suml_{f \in F_n W_r} f^{-1} \rho = G_{av}^{W_r} \rho + l_1,
\end{equation}
where the term~$l_1$ is bounded in absolute value by~$2^{-n}$. The last equality holds true due to the $F_n$-almost invariance of~$W_{r}$. 
Take~$\eps > 0$ satisfying 
$\mbbH_{4\sqrt{\eps}}(G_{av}^{F_n}\rho) \succsim \phi(n)$. For sufficiently large~$n$, the term~$l_1$ is negligible while computing $\eps$--entropy of~$G_{av}^{W_{r}} \rho$. Lemma~\ref{lm_lowerbound} gives
\begin{equation}
\mbbH_\eps(G_{av}^{W_{r}}\rho) \ge 
\mbbH_{4\eps}(G_{av}^{W_r} \rho + l_1) \ge
\mbbH_{4\sqrt{\eps}}(G_{av}^{F_n}\rho) \succsim \phi(n).
\end{equation}
Therefore, $G$~has a scaling entropy growth gap with respect to~$\{W_r\}$ and bound function $\psi(r) = \phi(n(r))$, where~$n(r)$ is the maximal~$n$ such that $k_n < r$.
\end{proof}

The fact that every compact representation decomposes into a direct  sum of finite-dimensional representations implies the absence of a free compact action of the infinite symmetric group~$S_\infty$. Indeed, the only finite-dimensional irreducible representations of~$S_\infty$ are the trivial and sign representations, which do not distinguish permutations with the same sign. 
This observation suggests the conjecture that $S_\infty$~should have a scaling entropy growth gap. It is unknown to the author whether this conjecture is true or not. 

\bibliographystyle{amsplain}

\end{document}